\documentclass{article}

\usepackage{cite}
\usepackage{amssymb}
\usepackage{amsmath}
\usepackage{amsfonts}
\usepackage{amsthm}
\usepackage[all]{xy}
\usepackage{enumerate}
\usepackage{colonequals}
\usepackage{mathrsfs} 
\usepackage[dvipsnames]{xcolor}
\usepackage{graphicx}
\usepackage{enumitem}
 \usepackage{mathtools}

\usepackage{url}

\usepackage[mathscr]{euscript}

\newtheorem{theorem}{Theorem}[section]
\newtheorem{proposition}{Proposition}[section]

\newtheorem*{theorem*}{Theorem}
\newtheorem{corollary}{Corollary}[section]
\newtheorem{definition}{Definition}[section]

\newtheorem{lemma}{Lemma}[section]
\newtheorem{claim}{Claim}[section]

\numberwithin{equation}{section}

\allowdisplaybreaks[2]

\allowdisplaybreaks

\begin{document}

\title{Hyperlinear approximations to amenable groups come from sofic approximations}
\author{Peter Burton}

\maketitle

\begin{abstract} We provide a quantitative formulation of the equivalence between hyperlinearity and soficity for amenable groups, showing that every hyperlinear approximation to such a group is essentially produced from a sofic approximation. This translates to a quantitative relationship between Hilbert-Schmidt and permutation stability for approximate homomorphisms which appropriately separate the elements of the group. \end{abstract}

\section{Introduction}

\subsection{Sofic and hyperlinear approximations}

Soficity and hyperlinearity are two ways of expressing the idea that a countable discrete group is a limit of finite approximations. In the sofic case, the relevant approximations are partial actions by permutations finite sets which model the action of the group on itself by left-translations. In the hyperlinear case, the relevant approximations are partial unitary representations on finite dimensional Hilbert spaces which model the left regular representations of the group. It is unknown whether every countable discrete group is sofic, and the same question is open for hyperlinearity. We refer the reader to the following surveys on these topics: \cite{pestov-sofic-survey, capraro-lupini}.\\
\\
We now present the relevant definitions. If $V$ is a finite set we write $\mathrm{Sym}(V)$ for the group of permutations of $V$.

\begin{definition} Let $G$ be a countable discrete group, let $F \subseteq G$ be finite and let $\epsilon > 0$. We define an $(F,\epsilon)$ \textbf{sofic approximation} to $G$ to be a finite set $V$ and a map $\sigma:G \to \mathrm{Sym}(V)$ such that the following hold. \begin{itemize} \item For all $g,h \in F$ we have \[ \frac{1}{|V|}|\{v \in V: \sigma(g)\sigma(h)v \neq \sigma(gh)v\}| \leq \epsilon  \] \item For every distinct pair $g,h \in F$ we have \[ \frac{1}{|V|}|\{v \in V:\sigma(g)v \geq \sigma(h)v\}| \geq 1- \epsilon \] \end{itemize} We define the group $G$ to be \textbf{sofic} if there exists a sequence $(\sigma_n)_{n=1}^\infty$ such that $\sigma_n$ is an $(F_n,\epsilon_n)$-sofic approximation to $G$ where $(F_n)_{n=1}^\infty$ is an increasing sequence of finite subsets of $G$ whose union is the entire group and $(\epsilon_n)_{n=1}^\infty$ is a decreasing sequence of positive numbers whose limit is zero. \end{definition}

If $X$ is a Hilbert space we write $\mathrm{U}(X)$ for the group of unitary operators on $X$. If $X$ is finite dimensional we write $\Delta(X)$ for the dimension of $X$ and in this case if $a$ is a linear operator on $X$ we define the Hilbert-Schmidt norm of $a$ by \[ ||a||_{\mathrm{HS}} = \sqrt{\frac{1}{\Delta(X)}\,\mathrm{trace}(a^\ast a)} \] 

\begin{definition} Let $G$ be a countable discrete group, let $F \subseteq G$ be finite and let $\epsilon > 0$. We define an $(F,\epsilon)$ \textbf{hyperlinear approximation} to $G$ to be a finite dimensional Hilbert space $X$ and a function $\alpha:G \to \mathrm{U}(X)$ such that the following hold. \begin{itemize} \item For all $g,h \in F$ we have \[ ||\alpha(gh) - \alpha(g)\alpha(h) ||_{\mathrm{HS}} \leq \epsilon \] \item For every distinct pair $g,h \in F$ we have \[ ||\alpha(g) - \alpha(h)||_{\mathrm{HS}} \geq \sqrt{2} - \epsilon \] \end{itemize} We define the group $G$ to be \textbf{hyperlinear} if there exists a sequence $(\alpha_n)_{n=1}^\infty$ such that $\alpha_n$ is an $(F_n,\epsilon_n)$ hyperlinear approximation to $G$ where $(F_n)_{n=1}^\infty$ is an increasing sequence of finite subsets of $G$ whose union is the entire group and $(\epsilon_n)_{n=1}^\infty$ is a decreasing sequence of positive numbers whose limit is zero. \end{definition}

There is a natural way to produce a hyperlinear approximation from a sofic approximation, as described below.

\begin{definition} \label{def.sofhyp} Let $G$ be a countable discrete group and let $\sigma:G \to \mathrm{Sym}(V)$ be an $(F,\epsilon)$ sofic approximation to $G$. We define a hyperlinear approximation $\alpha:G \to \mathrm{U}(X)$ to be \textbf{induced} by $\sigma$ if there exists an orthonormal basis $\{\zeta_v:v \in V\}$ for $X$ such that $\alpha(g)\zeta_v = \zeta_{\sigma(g)v}$ for all $g \in G$ and $v \in V$.\end{definition}

The above construction clearly implies that every sofic group is hyperlinear, but the validity of the converse implication is a long-standing problem in the field. (See, for example, Open Question 3.4 in \cite{pestov-sofic-survey} where the author speculates a positive answer.) The original article \cite{Gromov:1999aa} introducing soficitiy observed that every amenable group is sofic, so that among amenable groups soficity and hyperlinearity are trivially equivalent. Our goal in the present article is to prove the following theorem, which asserts that this equivalence holds not only on the qualititative level of the group but also a quantitative level among the approximations themselves. In other words, for an amenable group the procedure in Definition \ref{def.sofhyp} is essentially the only way to obtain a hyperlinear approximation.

\begin{theorem} \label{thm} Let $G$ be an amenable countable discrete group, let $F \subseteq G$ be finite and let $\epsilon > 0$. Then there exist a finite subset $K$ of $G$ and $\delta > 0$ depending only on $F$ and $\epsilon$ such that if $\alpha:G \to \mathrm{U}(X)$ is a $(K,\delta)$ hyperlinear approximation to $G$ then there exists a hyperlinear approximation $\omega:G \to \mathrm{U}(X)$ induced by a $(F,\epsilon)$ sofic approximation to $G$ with $||\alpha(g)-\omega(g)||_{\mathrm{HS}} \leq \epsilon$ for all $g \in F$. \end{theorem}

\subsection{Application to stability of approximations}

A natural question to ask about sofic or hyperlinear approximations is the extent to which they differ from approximations by genuine actions on finite sets or finite dimensional vector spaces. The next definitions provide a way of investigating this topic.

\begin{definition} Let $G$ be a countable discrete group. 

\begin{itemize}\item Let $\sigma = (\sigma_n:G \to \mathrm{Sym}(V_n))_{n \in \mathbb{N}}$ and $\tau = (\tau_n:G \to \mathrm{Sym}(V_n))_{n \in \mathbb{N}}$ be two sofic approximations to $G$ with the same underlying finite sets. We define $\sigma$ and $\tau$ to be a \textbf{asymptotic distance zero} if for every $g \in G$ we have \[ \lim_{n \to \infty} \frac{1}{|V_n|} |\{v \in V_n:\sigma(g)v \neq \tau(g)v\}| = 0 \] \item Let $\alpha = (\alpha_n:G \to \mathrm{U}(X_n))_{n \in \mathbb{N}}$ and $\beta = (\beta_n:G \to \mathrm{U}(X_n))_{n \in \mathbb{N}}$ be two hyperlinear approximations to $G$ with the same underlying Hilbert spaces. We define $\alpha$ and $\\beta$ to be at \textbf{asymptotic distance zero} if for every $g \in G$ we have \[ \lim_{n \to \infty} ||\alpha(g) - \beta(g)||_{\mathrm{HS}} = 0 \] \end{itemize} \end{definition}

\begin{definition} Let $G$ be a countable discrete group. \begin{itemize}\item Let $\sigma = (\sigma_n:G \to \mathrm{Sym}(V_n))_{n \in \mathbb{N}}$ be a hyperlinear approximation to $G$. We define $\sigma$ to be \textbf{perfect} if each $\sigma_n$ is a genuine group homomorphism. \item Let $\alpha = (\alpha_n:G \to \mathrm{U}(X_n))_{n \in \mathbb{N}}$ be a hyperlinear approximation to $G$. We define $\alpha$ to be \textbf{perfect} if each $\alpha_n$ is a genuine group homomorphism. \end{itemize} \end{definition}

\begin{definition} Let $G$ be a countable discrete group. \begin{itemize}\item We define $G$ to be \textbf{stably sofic} if every sofic approximation to $G$ is at asymptotic distance zero from a perfect sofic approximation to $G$. \item We define $G$ to be \textbf{stably hyperlinear} if every hyperlinear approximation to $G$ is at asymptotic distance zero from a perfect hyperlinear approximation to $G$. \end{itemize} \end{definition}

We note that above notions differ slightly from the concepts of `permutation stability' and `Hilbert-Schmidt stability' used in references such as \cite{MR3350728}, \cite{BECKER2020108298}, \cite{Hadwin2017StabilityOG}, \cite{levit_lubotzky_2022} and \cite{2022arXiv220602268L}, where the objects of study are `approximate homomorphisms' that correspond to dropping the second clause in our definitions of sofic and hyperlinear approximations. Stable soficity in our sense was studied in \cite{MR4105530} and the relationship with stable hyperlinearity is explored in \cite{2022arXiv221110492D}. 

\begin{theorem} \label{thm.2} Let $G$ be an amenable countable discrete group and suppose $G$ is stably sofic. Then $G$ is stably hyperlinear. \end{theorem}

\begin{proof}[Proof of Theorem \ref{thm.2} from Theorem \ref{thm}] Assume $G$ is permutation stable and let $(\alpha_n:G \to \mathrm{U}(X_n))_{n \in \mathbb{N}}$ be a hyperlinear approximation to $G$. Let $(F_n)_{n=1}^\infty$ be an increasing sequence of finite subsets of $G$ whose union is the entire group. For each $n \in \mathbb{N}$, Theorem \ref{thm} allows us to choose a finite subset $K_n$ of $G$ and $\delta_n > 0$ such that if $\beta:G \to \mathrm{U}(Y)$ is a $(K_n,\delta_n)$ hyperlinear approximation to $G$ then there exists a hyperlinear approximation $\gamma:G \to \mathrm{U}(Y)$ to $G$ induced by a $\left(F_n,\frac{1}{n}\right)$ sofic approximation to $G$ with $||\beta(g)-\gamma(g)||_{\mathrm{HS}} \leq \frac{1}{n}$ for all $g \in F_n$. \\
\\
Choose a sequence of natural numbers $m_1<m_2<\cdots$ which increases fast enough that $\alpha_k$ is a $(K_n,\delta_n)$ hyperlinear approximation to $G$ for all $k \in \{m_n,\ldots,m_{n+1}\}$. Then for all $n \in \mathbb{N}$ and all $k \in \{m_n,\ldots,m_{n+1}\}$ there exists an $\left(F_n,\frac{1}{n}\right)$ sofic approximation $\sigma_k:G \to \mathrm{Sym}(V_k)$ to $G$ inducing a hyperlinear approximation $\omega_k:G \to \mathrm{U}(X_k)$ such that \begin{equation} \label{eq.dark-1} ||\alpha_k(g)-\omega_k(g)||_{\mathrm{HS}} \leq \frac{1}{n} \end{equation} for all $g \in F_n$. Since the union of the $F_n$ is equal to $G$, the sequence $(\sigma_k:G \to \mathrm{Sym}(V_k))_{k \in \mathbb{N}}$ is a sofic approximation to $G$. Therefore permutation stability of $G$ implies that there exists a perfect sofic approximation $(\tau_k:G \to \mathrm{Sym}(V_k))_{k \in \mathbb{N}}$ with \begin{equation} \label{eq.dark-2} \lim_{n \to \infty} \frac{1}{|V_k|}|\{v \in V_k:\tau_k(g)v \neq \sigma_k(g)v\|| =0\end{equation} for all $g \in G$. The sofic approximation $\tau_k:G \to \mathrm{Sym}(V_k)$ induces a perfect hyperlinear approximation to $G$ by permuting the same orthonormal basis for $X_k$ as did $\sigma_k$. From (\ref{eq.dark-1}) and (\ref{eq.dark-2}) we see that $\lim_{n \to \infty}||\alpha_k(g)-\omega_k(g)||_{\mathrm{HS}} = 0$ as required. \end{proof}

\subsection{Acknowledgements}

The author thanks Alex Lubotzky for raising the question of whether Theorem \ref{thm.2} is true and Lewis Bowen for helpful comments on the manuscript.

\section{Random vectors and Hilbert-Schmidt norms}

\begin{proposition} \label{prop.09-11.3} Let $X$ be a finite dimensional Hilbert space, let $a$ be a linear operator on $X$ and let $\boldsymbol{\xi}$ be a uniform random element of the unit sphere of $X$. Then we have $\mathbb{E}[||a\boldsymbol{\xi}||^2] = ||a||_{\mathrm{HS}}^2$ \end{proposition}

\begin{proof}[Proof of Proposition \ref{prop.09-11.3}] Write $d = \Delta(X)$ and let $\zeta_1,\ldots,\zeta_d$ be an orthonormal basis for $X$. According to the singular value decomposition we can find an operator $b$ which is diagonal in this basis and unitary operators $u$ and $v$ such that $v^\ast a u = b$. We have \begin{equation} \label{eq.09-11.8} \mathbb{E}[||a \boldsymbol{\xi}||^2] = \mathbb{E}[||au\boldsymbol{\xi}||^2] = \mathbb{E}[||v^{\ast}au \boldsymbol{\xi}||^2] = \mathbb{E}[||b\boldsymbol{\xi}||^2] \end{equation} where the first equality in the previous display holds since the distribution of $\boldsymbol{\xi}$ is $u$-invariant and the second equality holds since the norm $|| \cdot||$ is $v^{\ast}$-invariant.\\
\\
Write $s_1,\ldots,s_d$ for the diagonal entries of $b$, which are the singular values of $a$. Using (\ref{eq.09-11.8}) we have \begin{align} \mathbb{E}[||a\boldsymbol{\xi}||^2] &= \mathbb{E}[||b\boldsymbol{\xi}||^2] \label{eq.09-16.1} \\ & = \mathbb{E}\left[ \sum_{k=1}^d |\langle b\boldsymbol{\xi},\zeta_k\rangle|^2 \right] \nonumber \\ & = \sum_{k=1}^d \mathbb{E}[|\langle b \boldsymbol{\xi},\zeta_k\rangle|^2] \label{eq.09-16.2} \\ & = \sum_{k=1}^d \mathbb{E}[|\langle \boldsymbol{\xi},b\zeta_k\rangle|^2] \label{eq.09-16.3} \\ & = \sum_{k=1}^d \mathbb{E}[|\langle \boldsymbol{\xi},s_k\zeta_k\rangle|^2] \nonumber \\ & = \sum_{k=1}^d |s_k|^2\mathbb{E}[|\langle \boldsymbol{\xi},\zeta_k\rangle|^2] \label{eq.well} \end{align} Here, (\ref{eq.09-16.3}) follows from (\ref{eq.09-16.2}) since the diagonal matrix $b$ in the singular value decomposition of $a$ has nonnegative entries.\\ 
\\
Since the distribution of $\boldsymbol{\xi}$ is uniform, we have $\mathbb{E}[|\langle \boldsymbol{\xi},\kappa\rangle|] = \mathbb{E}[|\langle \boldsymbol{\xi},\eta\rangle|]$ for all unit vectors $\kappa,\eta \in X$. Since we have \[ 1 = \mathbb{E}\left[ \sum_{k=1}^d |\langle \boldsymbol{\xi},\zeta_k\rangle|^2 \right] \] it follows that $\mathbb{E}[|\langle \boldsymbol{\xi},\zeta_k\rangle|^2] = \frac{1}{d}$ for all $k \in \{1,\ldots,d\}$ Thus from (\ref{eq.well}) we obtain \[ \mathbb{E}[||a\boldsymbol{\xi}||^2] = \frac{1}{d}\sum_{k=1}^d |s_k|^2  \] Since $s_1,\ldots,s_d$ are the singular values of $a$, the last expression is equal to $||a||_{\mathrm{HS}}^2$ and we obtain $\mathbb{E}[||a\boldsymbol{\xi}||^2] = ||a||_{\mathrm{HS}}^2$. \end{proof}

By applying Jensen's inequality to Proposition \ref{prop.09-11.3} we obtain $\mathbb{E}[||a\boldsymbol{\xi}||] \leq ||a||_{\mathrm{HS}}$ and then from Markov's inequality we obtain the following corollary.

\begin{corollary}\label{cor.09-12.1} Let $X$ be a finite dimensional Hilbert space, let $a$ be a linear operator on $X$ and let $\boldsymbol{\xi}$ be a uniform random element of the unit sphere of $X$. Then for any $c>0$ we have \[ \mathbb{P}\left[ ||a\boldsymbol{\xi}|| > c||a||_{\mathrm{HS}} \right] \leq \frac{1}{c}      \] \end{corollary}

We may also observe that if $u_1,\ldots,u_n$ are unitary operators on $X$ then the distribution of $u_k \boldsymbol{\xi}$ is the same as the distribution of $\boldsymbol{\xi}$ for all $k \in \{1,\ldots,n\}$. Therefore we have \[ \mathbb{E}\left[ \frac{1}{n} \sum_{k=1}^n ||au_k\boldsymbol{\xi}||^2 \right] = \frac{1}{n} \sum_{k=1}^n \mathbb{E}[||au_k\boldsymbol{\xi}||^2]= \frac{1}{n} \sum_{k=1}^n \mathbb{E}[||a\boldsymbol{\xi}||^2] = ||a||_{\mathrm{HS}}^2 \] Therefore Jensen's inequality implies \[ \frac{1}{n} \sum_{k=1}^n\mathbb{E}[||au_k\boldsymbol{\xi}||] \leq ||a||_{\mathrm{HS}} \]

By applying Markov's inequality to the last display we obtain the following additional corollary.

\begin{corollary} \label{cor.09-12.2} Let $X$ be a finite dimensional Hilbert space, let $p$ be a projection on $X$ and let $u_1,\ldots,u_n$ be unitary operators on $X$. Then for any $c>0$ we have \[ \mathbb{P}\left[ \frac{1}{n} \sum_{k=1}^n ||p u_k \boldsymbol{\xi}||^2 > c\frac{ \mathrm{tr}(p)}{\Delta(X)} \right] \leq \frac{1}{c} \] \end{corollary}

In the above corollary and throughout the article, we implicitly assume that the term `projection' refers exclusively to orthogonal projections. If $p$ and $q$ are projections, we write $p \vee q$ for the minimal projection such that $p \vee q \geq p$ and $p \vee q \geq q$. We may define $p \vee q$ explicitly as $p+(I-p)q(I-p) = q+(I-q)p(I-q)$. Note that for any vector $\xi$ we have \[ ||(p \vee q)\xi|| \leq ||p\xi||+||q\xi|| \]

\section{Preliminary results on hyperlinear approximations} \label{sec.prelim}

Throughout Section \ref{sec.prelim} we fix a nonzero finite dimensional Hilbert space $X$, a countable discrete group $G$ and a map $\alpha:G \to \mathrm{U}(X)$.

\subsection{Orthogonalizing approximately orthogonal vectors}

Write $C_n = \sqrt{n}(8n)^{2n}$

\begin{proposition}\label{prop.09-15.1} Let $n \in \mathbb{N}$, let $\lambda > 0$ and let $\xi_1,\ldots,\xi_n$ be a family of unit vectors in $X$ such that $|\langle \xi_j,\xi_k \rangle| \leq \lambda$ for all $j,k \in \{1,\ldots,n\}$. For $j \in \{1,\ldots,n\}$ write $p_j$ for the projection onto the span of $\xi_1,\ldots,\xi_j$. Then for $j \in \{1,\ldots,n-1\}$ we have $||p_j \xi_{j+1}|| \leq C_n\lambda$. \end{proposition}

\begin{proof}[Proof of Proposition \ref{prop.09-15.1}] We can perform the Gram-Schmidt procedure to define a family of vectors $\zeta_1 = \xi_1$ and \[ \zeta_j = \frac{\xi_j - \langle \xi_j,\zeta_{j-1}\rangle\zeta_{j-1} - \cdots - \langle \xi_j,\zeta_1\rangle \zeta_1}{||\xi_j - \langle \xi_j,\zeta_{j-1}\rangle\zeta_{j-1} - \cdots - \langle \xi_j,\zeta_1\rangle \zeta_1||}  \] for $j \in \{2,\ldots,n\}$. We make the following claim.

\begin{claim} \label{claim.1} For all $j \in \{1,\ldots,n\}$ we have $||\zeta_j - \xi_j|| \leq (8n)^{2j}\lambda$. \end{claim}

We verify Claim \ref{claim.1} by induction on $j$. The claim is trivial for $j = 1$, so let $j \in \{2,\ldots,n\}$ and assume we have shown that $||\zeta_k - \xi_k|| \leq \lambda (8n)^k$ for all $k \in \{1,\ldots,j-1\}$. Then for $k \in \{1,\ldots,j-1\}$ we have \begin{equation} \label{eq.09-15.1} |\langle \xi_j, \zeta_k \rangle| \leq ||\zeta_k - \xi_k|| + |\langle \xi_j, \xi_k \rangle| \leq \lambda (8n)^k + \lambda \leq 2 (8n)^k \lambda \end{equation} Therefore \[ ||\langle \xi_j,\zeta_{j-1}\rangle \zeta_{j-1}+\cdots+\langle \xi_j,\zeta_1 \rangle \zeta_1|| \leq |\langle \xi_j,\zeta_{j-1}\rangle|+\cdots+|\langle \xi_j,\zeta_1\rangle| \leq 2(j-1)(8n)^{j-1} \lambda \]

Write \[ t_j = ||\xi_j - \langle \xi_j,\zeta_{j-1}\rangle\zeta_{j-1} - \cdots - \langle \xi_j,\zeta_1\rangle \zeta_1|| \] so that \[ 1 \geq t_j \geq 1- 2(j-1)(8n)^{j-1}\lambda \]

It follows that \begin{align*} ||\zeta_j - \xi_j|| &= \left \vert \left \vert \xi_j -\frac{\xi_j - \langle \xi_j,\zeta_{j-1}\rangle\zeta_{j-1} - \cdots - \langle \xi_j,\zeta_1\rangle \zeta_1}{t_j} \right \vert \right \vert \\ & = \frac{1}{t_j}||t_j\xi_j - \xi_j + \langle \xi_j,\zeta_{j-1}\rangle\zeta_{j-1} + \cdots + \langle \xi_j,\zeta_1\rangle \zeta_1|| \\ & \leq \frac{2(j-1)(8n)^{j-1}\lambda}{1-2(j-1)(8n)^{j-1}\lambda} \leq 8(j-1)^2(8n)^{2j-2} \lambda \leq (8n)^{2j}\lambda \end{align*} This verifies Claim \ref{claim.1}. Therefore (\ref{eq.09-15.1}) is valid for all $k \in \{1,\ldots,n\}$ and so we obtain \[ ||p_j\xi_{j+1}||^2 = |\langle \xi_{j+1},\zeta_j\rangle|^2 + \cdots + |\langle \xi_{j+1},\zeta_1\rangle|^2 \leq n(8n)^{4n}\lambda^2 \] This completes the proof of Proposition \ref{prop.09-15.1}. \end{proof}

\subsection{Orthogonalization in hyperlinear approximations}

Let $M$ be a finite subset of $G$ and let $\lambda > 0$. Consider the following pair of conditions on a vector $\xi \in X$.

\begin{description} \item[Condition $(M,\lambda)$-I] For all distinct pairs $g,h \in M$ we have $|\langle \alpha(g)\xi,\alpha(h)\xi \rangle| \leq \lambda$. \item[Condition $(M,\lambda)$-II] For all $g,h \in M$ we have $||(\alpha(gh) - \alpha(g)\alpha(h))\xi|| \leq \lambda$. \end{description}

\begin{proposition} \label{prop.09-13.1} Let $\rho \geq 0$ and $\lambda > 0$, let $M$ be a finite subset of $G$ and let $\xi \in X$ be a unit vector satisfying Condition $(M,\lambda)$-I.  Let $p$ be a projection on $X$ such that \begin{equation} \label{eq.09-13.1} \frac{1}{|M|} \sum_{g \in M}||p\alpha(g)\xi|| \leq \rho \end{equation} Then there exists a set $M' \subseteq M$ with $|M'| \geq (1-\sqrt{\rho})|M|$ and an orthonormal basis $(\zeta_g)_{g \in M'}$ for the span of $\{(I-p)\alpha(g)\xi:g \in M\}$ such that \[ \frac{1}{|M|} \sum_{g \in M}||\alpha(g)\xi - \zeta_g|| \leq 4(\sqrt{\rho} + C_{|M|}\lambda) \] \end{proposition}

\begin{proof}[Proof of Proposition \ref{prop.09-13.1}] Enumerate $M = \{g_1,\ldots,g_n\}$. For $j \in \{1,\ldots,n\}$ define $\xi_j = \alpha(g_j)\xi$ and let $q_j$ be the orthogonal projection onto the span of $\xi_1,\ldots,\xi_j$. Since Condition $(M,\lambda)$-I is satisfied, Proposition \ref{prop.09-15.1} implies that $||q_j \alpha(g_{j+1})\xi||^2 \leq C_n\lambda$ for all $j \in \{1,\ldots,n-1\}$. Hence for all $j \in \{2,\ldots,n\}$ we have \[ ||(p \vee q_{j-1})\xi_j|| \leq C_n\lambda+||p \xi_j||  \]

The inequality (\ref{eq.09-13.1}) implies that there exists a set $S \subseteq \{1,\ldots,n\}$ such that $|S| \geq (1-\sqrt{\rho})n$ and $||p\xi_j|| \leq \sqrt{\rho}$ for all $j \in S$. For $j \in \{2,\ldots,n\}$ define \[ \zeta_{g_j} = \frac{(I-(p \vee q_{j-1}))\xi_j}{||(I-(p\vee q_{j-1}))\xi_j||} \] If $j \in S$ then we have \[ ||\zeta_{g_j} - \alpha(g_j)\xi|| \leq \frac{\sqrt{\rho}+C_n\lambda}{1-\sqrt{\rho}-C_n\lambda} \leq 2( \sqrt{\rho}+C_n\lambda) \] and if $j \notin S$ then we still have the trivial bound $||\zeta_{g_j}-\alpha(g_j)\xi|| \leq 2$. Thus we may compute \begin{align*} \frac{1}{n} \sum_{j=1}^n ||\zeta_j - \xi_j|| & \leq \frac{1}{n} \sum_{j \in S} ||\zeta_{g_j}-\xi_j|| + \frac{1}{n} \sum_{\substack{1 \leq j \leq n \\ j \notin S}} ||\zeta_{g_j}-\xi_j|| \\ & \leq \frac{1}{n} \sum_{j \in S} 2(\sqrt{\rho}+C_n\lambda) + \frac{2}{n}|S| \leq 4(\sqrt{\rho}+C_n\lambda) \end{align*} We can let $M' = \{\zeta_{g_j}:j \in S\}$. \end{proof}

\subsection{Proximity to permutation actions} \label{sec.good}

Throughout the remainder of Section \ref{sec.prelim} we fix the finite set $M$. Suppose furthermore that $L$ is a finite subset of $M$ such that $|hM \cap M| \geq (1-\eta)|M|$ for some $\eta > 0$ and all $h \in L$. Then for each $h \in L$ we can find a permutation $\varsigma(h)$ of $M$ such that \[ \frac{1}{|M|}|\{g \in M:\varsigma(h)g \neq hg\}| \leq \eta \]

\begin{proposition} \label{prop.09-15.2} Let $\varpi,\lambda > 0$. Let $\xi \in X$ satisfy Condition $(M,\lambda)$-II and let $(\zeta_g)_{g \in M}$ be a family of unit vectors indexed by $M$ such that \begin{equation} \label{eq.09-17.1} \frac{1}{|M|} \sum_{g \in M} ||\alpha(g)\xi-\zeta_g|| \leq \varpi \end{equation} Then for any $h \in L$ we have \[ \frac{1}{|M|} \sum_{g \in M}||\alpha(h)\zeta_g - \zeta_{\varsigma(h)g}|| \leq \eta+2\varpi+\lambda \] \end{proposition}

\begin{proof}[Proof of Proposition \ref{prop.09-15.2}] Since $h \in L$ we have \[ |\{g \in M:\varsigma(h)g \neq hg\}| \leq \eta |M| \] and therefore \[ \frac{1}{|M|} \sum_{g \in M} ||\alpha(hg)\xi - \alpha(\varsigma(h)g)\xi|| \leq \eta \] From Condition $(M,\lambda)$-II we have \[ \frac{1}{|M|} \sum_{g \in M} ||\alpha(hg)\xi - \alpha(h)\alpha(g)\xi|| \leq \lambda \] By putting the two previous displays together with (\ref{eq.09-17.1}) we obtain 

\begin{align*} \frac{1}{|M|} \sum_{g \in M}||\alpha(h)\zeta_g - \zeta_{\varsigma(h)g}|| & \leq \frac{1}{|M|} \sum_{g \in M}||\alpha(h)\zeta_g - \alpha(h)\alpha(g)\xi|| + \frac{1}{|M|} \sum_{g \in M}||\alpha(h)\alpha(g)\xi-\alpha(gh)\xi|| \\ & \hspace{0.5 in} +\frac{1}{|M|} \sum_{g \in M} ||\alpha(gh)\xi-\alpha(\varsigma(h)g)\xi|| +\frac{1}{|M|} \sum_{g \in M} ||\alpha(\varsigma(h)g)\xi-\zeta_{\varsigma(h)g}|| \\ & \leq \eta+\lambda+2\varpi \end{align*} \end{proof}

\subsection{Approximation invariance of projections onto partial orbits}

\begin{proposition} \label{prop.09-15.6} Let $\xi_1,\ldots,\xi_m \in X$ be unit vectors satisfying Condition $(M,\lambda)$-II and let $q$ be the projection onto the span of $\{\alpha(g)\xi_j:1 \leq j \leq m,g \in M\}$. Then we have \[ ||(I-q)\alpha(h)q||_{\mathrm{HS}} \leq \sqrt{\frac{m|M|(\eta+5C_{|M|}\lambda)}{\Delta(X)}} \] for all $h \in L$.  \end{proposition}

\begin{proof}[Proof of Proposition \ref{prop.09-15.6}] Fix $j \in \{1,\ldots,m\}$. By applying Proposition \ref{prop.09-13.1} with $p=0$ and $\rho = 0$ we can find an orthonormal basis $(\zeta_{j,g})_{g \in M}$ for the span of $\{\alpha(g)\xi_j:g \in M\}$ such that \begin{equation} \label{eq.09-15.8} \sum_{g \in M}||\alpha(g)\xi_j-\zeta_{g,j}|| \leq 4C_{|M|}\lambda |M| \end{equation} Therefore we may compute \begin{align} &||(I-q)\alpha(h)q||^2_{\mathrm{HS}} \nonumber \\ & \leq \frac{1}{\Delta(X)} \sum_{j=1}^m \sum_{g \in M} ||(I-q)\alpha(h)\zeta_{g,j}||^2 \nonumber \\ & \leq \frac{1}{\Delta(X)} \sum_{j=1}^m \sum_{g \in M} \Bigl(||(I-q)\alpha(hg)\xi_j||+||(I-q)(\alpha(h)\alpha(g)-\alpha(hg))\xi_j||+||(I-q)\alpha(h)(\alpha(g)\xi_j-\zeta_{g,j})||\Bigr)^2 \nonumber \\ & \leq \frac{3}{\Delta(X)} \sum_{j=1}^m\sum_{g \in M} \Bigl(||(I-q)\alpha(hg)\xi_j||^2+||(I-q)(\alpha(h)\alpha(g)-\alpha(hg))\xi_j||^2+||(I-q)\alpha(h)(\alpha(g)\xi_j-\zeta_{g,j})||^2 \Bigr) \nonumber \\ & \leq \frac{3}{\Delta(X)} \sum_{j=1}^m\sum_{g \in M}||(I-q)\alpha(hg)\xi_j|| \label{eq.09-18.1} \\ & \hspace{0.5 in} + \frac{3}{\Delta(X)} \sum_{j=1}^m \sum_{g \in M} ||(\alpha(h)\alpha(g)-\alpha(hg))\xi_j|| \label{eq.09-18.2} \\ & \hspace{1 in} +\frac{3}{\Delta(X)}\sum_{j=1}^m\sum_{g \in M} ||\alpha(g)\xi_j-\zeta_{g,j}|| \label{eq.09-18.3} \end{align}

Now we may observe the following bounds, from which Proposition \ref{prop.09-15.6} follows immediately.

\begin{itemize} \item If $g \in h^{-1}M \cap M$ then we have $(I-q)\alpha(gh)\xi =0$. Therefore the sum over $g \in M$ in (\ref{eq.09-18.1}) is at most $|h^{-1}M \setminus M| \leq \eta|M|$. \item The sum over $g \in M$ in (\ref{eq.09-18.2}) is at most $\lambda|M|$ since Condition $(M,\lambda)$-II holds for $\xi$ \item The sum over $g \in M$ in (\ref{eq.09-18.3}) is at most $4C_{|M|}\lambda|M|$ by (\ref{eq.09-15.8}). \end{itemize} \end{proof}

\begin{proposition} \label{prop.09-15.3} Let $\theta > 0$, let $p$ be a projection and let $u$ be a unitary operator such that $||(I-p)up||^2_{\mathrm{HS}} \leq \theta$. Then there exists an operator $v$ commuting with $p$ such that $v^\ast v = vv^\ast = p$ and $||(u-v)p||_{\mathrm{HS}}^2 \leq 4 \theta$.\end{proposition}

\begin{proof}[Proof of Proposition \ref{prop.09-15.3}] Define $w = p u p$, so that $w$ commutes with $p$. We have \[ ||(w-u)p||^2_{\mathrm{HS}} = || pup-up  ||^2_{\mathrm{HS}} = ||(I-p)up||^2_{\mathrm{HS}} \leq \theta \] Moreover, we have \begin{equation} \label{eq.09-15.3} ||w^\ast w - p||_{\mathrm{HS}} = ||pu^\ast p u p - p||_{\mathrm{HS}} \leq ||p - pu^\ast u p||_{\mathrm{HS}} + || pu^\ast(I-p)up ||_{\mathrm{HS}} = ||(I-p)up||_{\mathrm{HS}} \leq \theta \end{equation} Now, choose an orthonormal basis $\zeta_1,\ldots,\zeta_d$ for $X$ such that $\zeta_1,\ldots,\zeta_j$ is an orthonormal basis for the range of $p$. Since $p$ and $w$ commute, according to the singular value decomposition we can find unitary operators $a$ and $b$ such that $a^\ast p b$ and $a^\ast w b$ are both diagonal in this basis. Let $s_1,\ldots,s_j$ be the nonzero entries of $a^\ast w b$. Since $p\zeta_k = \zeta_k$ for $k \in \{1,\ldots,j\}$ we have that the entries of $a^\ast p b$ are one in the first $j$ rows. Therefore (\ref{eq.09-15.3}) implies that \[ \Bigl \vert |s_1|^2 - 1 \Bigr \vert + \cdots \Bigl \vert |s_j|^2-1 \Bigr \vert \leq \theta \Delta(X)  \] Therefore there exists a set $S \subseteq \{1,\ldots,j\}$ with $|S| \geq j-\theta \Delta(X)$ such that  \[ \Bigl \vert |s_k|^2 -1 \Bigr \vert \leq \theta \] for all $k \in S$ and therefore \[ \left \vert s_k - \frac{s_k}{|s_k|} \right \vert \leq \frac{|s_k^2 - s_k|}{|s_k|} \leq \frac{|s_k^2-1|}{|s_k|} \leq \frac{\theta}{1-\theta} \leq 2\theta \] since $0 \leq s_k \leq 1$.

Thus if we let $c$ be a diagonal matrix with $\frac{s_k}{|s_k|}$ in the $k^{\mathrm{th}}$ row for all $k \in \{1,\ldots,j\}$ and $0$ in the $j+1,\ldots,d$ rows we find \[ \Delta(X)||c-a^\ast w b||^2_{\mathrm{HS}} \leq 2(j-|S|) + \sum_{k \in S} \left \vert s_k - \frac{s_k}{|s_k|} \right \vert \leq 4\theta \Delta(X) \] Then $c^\ast c = c c^\ast = b^\ast p b$ so that if we let $v = a c b^\ast$ then $v$ is as required. \end{proof}

\section{Statement and proof of main lemma}

\begin{lemma} \label{lem} Let $G$ be a countable discrete group, let $L \subseteq G$ be finite and let $\eta > 0$. Let $M$ be a finite subset of $G$ such that $|hM \cap M| \geq (1-\eta)|M|$ for all $h \in L$. Let also $\lambda,\kappa \in (0,1)$. Let $X$ be a finite dimensional Hilbert space and let $\alpha:G \to \mathrm{U}(X)$ be an $\left(M,\frac{\lambda}{8|M|^2}\right)$-hyperlinear approximation to $G$. Then we have decomposition $X = Y \oplus Z$ such that $\Delta(Z) \geq \frac{\kappa}{2}\Delta(X)$ and the following objects exist. \begin{itemize} \item A hyperlinear approximation $\gamma:G \to \mathrm{U}(Z)$ induced by an $(L,\kappa+\eta)$ sofic approximation to $G$ such that for all $h \in L$ we have \[ \frac{1}{\Delta(Z)} \mathrm{tr}((\gamma(h)-\alpha(h))^\ast(\gamma(h)-\alpha(h))  \leq 2\eta+4\kappa+5C_{|M|}\lambda \] \item A $(L,4\sqrt{\eta+5C_{|M|}\lambda})$ hyperlinear approximation $\beta:G \to \mathrm{U}(Y)$ such that for all $h \in L$ we have \[ \frac{1}{\Delta(Y)} \mathrm{tr}((\gamma(h)-\alpha(h))^\ast(\gamma(h)-\alpha(h)) \leq 16(\eta+5C_{|M|}\lambda) \] \end{itemize}  \end{lemma}

The key point in the above lemma is that $\kappa$ does not appear in the bounds involving $\beta$. This will allow us to apply the lemma roughly $\frac{2}{\kappa}$ times to allow the accumulated spaces $Z$ to almost fill out $X$. 

\begin{proof}[Proof of Lemma \ref{lem}]

Let $X$ be a finite dimensional Hilbert space and let $\alpha:G \to \mathrm{U}(X)$ be an $\left(M,\frac{\lambda}{8|M|^2}\right)$-hyperlinear approximation to $G$. Also let $\boldsymbol{\xi}$ be a uniform random element of the unit sphere of $X$. \\
\\
Since $\alpha$ is an $(M,\frac{\lambda}{8})$-hyperlinear approximation to $G$, Proposition \ref{prop.09-11.3} implies that we have that for each distinct pair $g,h \in M$ we have \[ \sqrt{2}+\frac{\lambda}{8|M|^2} \geq \mathbb{E}[||\alpha(g)\boldsymbol{xi} - \alpha(h)\boldsymbol{\xi}||^2] \geq \sqrt{2} -\frac{\lambda}{8|M|^2}  \] and therefore \[ \mathbb{E}[|\langle \alpha(g)\boldsymbol{\xi},\alpha(h)\boldsymbol{\xi}\rangle|] \leq \frac{\lambda}{8|M|^2} \] By Markov's inequality we find \[ \mathbb{P}[|\langle\alpha(g)\boldsymbol{\xi},\alpha(g)\boldsymbol{\xi}\rangle| \leq \lambda] \geq 1- \frac{1}{8|M|^2} \] Similarly, for each pair $g,h \in M$ we have \[ \mathbb{P}[||\alpha(gh)\boldsymbol{\xi}-\alpha(g)\alpha(h)\boldsymbol{\xi}|| \leq \lambda] \geq 1- \frac{1}{8|M|^2} \] By intersecting the above sets over all pairs $g,h \in M$ we obtain \begin{equation} \label{eq.prob} \mathbb{P}[\,\mbox{Condition }(M,\lambda)-\mathrm{I}\mbox{ or }(M,\lambda)-\mathrm{II}\mbox{ fails for }\boldsymbol{\xi}] \leq \frac{1}{4} \end{equation} We now perform a recursive construction. In the initial stage, we let $\xi_1 \in X$ be any unit vector satisfying conditions $(M,\lambda)$-I and $(M,\lambda)$-II. We also let $p_1$ denote the projection onto the span of $\{\alpha(g)\xi_1:g \in F\}$.\\
\\
Now, let $n \geq 1$ and consider the $(n+1)^{\mathrm{st}}$ stage. \begin{description} \item[(Recursive hypothesis)] Suppose we have unit vectors $\xi_1,\ldots,\xi_n \in X$ satisfying Condition $(M,\lambda)$-I and $(M,\lambda)$-II such that if we write $p_k$ for the projection onto the span of $\{\alpha(g)\xi_k:g \in M\}$ and $q_k = p_1\vee \cdots \vee p_k$ then for all $k \in \{1,\ldots,n-1\}$ we have \[ \frac{1}{|M|} \sum_{g \in M} ||q_k\alpha(g)\xi_{k+1}|| \leq \kappa \] 
\item[(Termination condition)] If we have $\mathrm{tr}(q_n) > \frac{\kappa}{2} \Delta(X)$ then we may terminate the construction. 
\item[(Extension procedure)]Suppose $\mathrm{tr}(q_n) \leq \frac{\kappa}{2} \Delta(X)$. We may apply Corollary \ref{cor.09-12.2} to obtain \[ \mathbb{P}\left[ \frac{1}{|M|} \sum_{g \in M} ||q_n \alpha(g) \boldsymbol{\xi}|| > \frac{2}{\Delta(X)}\mathrm{tr}(q_n) \right] \leq \frac{1}{2} \] and therefore \[ \mathbb{P}\left[ \frac{1}{|M|} \sum_{g \in M} ||q_n \alpha(g) \boldsymbol{\xi}|| \leq \kappa \right] \geq \frac{1}{2} \] By combining the previous display with (\ref{eq.prob}) we have \[ \mathbb{P}\left[\mbox{ Conditions }(M,\lambda)-\mathrm{I}\mbox{ and }(M,\lambda)-\mathrm{II}\mbox{ hold for }\boldsymbol{\xi}\mbox{ and } \frac{1}{|M|} \sum_{g \in M} ||q_n \alpha(g) \boldsymbol{\xi}|| \leq \kappa \right] \geq \frac{1}{4} \] Then we can choose an element of the nonempty set in the previous display to serve as $\xi_{n+1}$. \end{description} 

The output of the construction is a sequence of vectors $\xi_1,\ldots,\xi_m$ satisfying the recursive hypothesis such that $\mathrm{tr}(q_m) > \frac{\kappa}{2}\Delta(X)$. We now perform another recursive construction. At the first stage, we apply Proposition \ref{prop.09-13.1} to $\xi_1$ with $\rho=0$ to obtain an orthonormal family of vectors $(\zeta_{1,g})_{g \in F}$ such that \[ \frac{1}{|M|} \sum_{g \in M} ||\zeta_{1,g}-\alpha(g)\xi_1|| \leq 4C_{|M|}\lambda \] We now let $k \in \{1,\ldots,m-1\}$ and consider the $(k+1)^{\mathrm{st}}$ stage.

\begin{description} \item[(Recursive hypothesis)] Suppose that for all $j \in \{1,\ldots,k\}$ we have a subset of $M_j$ of $M$ satisfying $|M_j| \geq (1-\kappa)|M|$, an orthonormal family of vectors $\{\zeta_{g,j}:1 \leq j \leq k, g \in M_j\}$ whose span is equal to the range of $q_k$ and such that for all $j \in \{1,\ldots,k\}$ we have \[ \frac{1}{|M_j|} \sum_{g \in M_j} ||\zeta_{g,j}-\alpha(g)\xi_j|| \leq \kappa+ 4C_{|M|}\lambda \] \item[(Extension procedure)] Since the our construction of $\xi_{k+1}$ implies that we have \[ \frac{1}{|M|} \sum_{g \in M}||q_k\alpha(g)\xi_{k+1}|| \leq \kappa \] we can apply Proposition \ref{prop.09-13.1} with $\rho = \kappa$ to obtain a set $M_{k+1} \subseteq M$ with $|M_{k+1}| \geq (1-\kappa)|M|$ and an orthonormal family of vectors $(\zeta_{g,k+1})_{g \in M_{k+1}}$ such that \[ \frac{1}{|M_j|} \sum_{g \in M_j}||\alpha(g)\xi_{k+1} - \zeta_{g,k+1}|| \leq 4(\kappa+C_{|M|}\lambda) \] and such that $q_k\zeta_{g,k+1} = 0$ for all $g \in M_{k+1}$. The last condition implies that $\{\zeta_{g,j}:g \in F, 1 \leq j \leq k+1\}$ is again an orthonormal family and so the recursive hypothesis is again satisfied. \end{description}

The output of the construction is an orthonormal family of vectors $\{\zeta_{g,j}:1 \leq j \leq m, g \in M_j\}$ whose span has dimension at least $\frac{\kappa}{2}\Delta(X)$ and such that \begin{equation} \label{eq.09-18.5} \frac{1}{|M_j|} \sum_{g \in M} ||\alpha(g)\xi_j - \zeta_{g,j}|| \leq 4(\kappa+C_{|M|}\lambda) \end{equation} for all $j \in \{1,\ldots,m\}$. Furthermore, the the span of $\{\zeta_{g,j}:1 \leq j \leq m, g \in M_j \}$ is equal to the span of \begin{equation} \label{eq.09-18.4} \{\alpha(g)\xi_j:1 \leq j \leq m,g \in M\} \end{equation} and Conditions $(M,\lambda)$-I and $(M,\lambda)$-II are satisfied for each $\xi_j$. Thus we may apply Proposition \ref{prop.09-15.2} to find that if we write $p$ for the orthogonal projection on the span of the set in (\ref{eq.09-18.4}) then we have \[ ||(I-p)\alpha(h)p||_{\mathrm{HS}} \leq  \sqrt{\frac{m|M|(\eta+5C_{|M|}\lambda)}{\Delta(X)}} \] for all $h \in L$. Therefore by applying \ref{prop.09-15.3} to $I-p$ for $\alpha(h)$ for each $h \in L$ we obtain an operator $\beta(h)$ which commutes with $I-p$ and satisfies $\beta(h)^\ast \beta(h) = I-p$ and \begin{equation} ||(\beta(h)-\alpha(h))(I-p)||_{\mathrm{HS}} \leq 4\sqrt{\frac{m|M|(\eta+5C_{|M|}\lambda)}{\Delta(X)}} \leq 4\sqrt{\eta+5C_{|M|}\lambda}  \label{eq.09-19.8} \end{equation} Now, let $h \in L$. Since $|M_j| \geq (1-\kappa)|M|$ and $|hM \cap M| \geq (1-\eta)|M|$ we have  \[ |hM_j \cap M_j| \geq (1-\kappa-\eta)|M_j| \] Therefore we can find a permutation $\varsigma_j(h)$ of $M_j$ such that \[ |\{g \in M_j:\varsigma_j(h)g \neq hg\}| \leq (\kappa+\eta)|M_j| \] By applying Proposition \ref{prop.09-15.2} to (\ref{eq.09-18.5}) with $\varpi = 4(\kappa+C_{|M|}\lambda)$ we obtain \[ \frac{1}{|M_j|} \sum_{g \in M_j} ||\alpha(h)\zeta_{j,g} - \zeta_{j,\varsigma_j(h)g}|| \leq 2\eta + 4\kappa+5C_{|M|}\lambda \] Define $\gamma(h)$ to be the operator $\bigoplus_{j=1}^m \varsigma_j(h)$. Then both $\beta(h)$ and $\gamma(h)$ commute with $p$. Furthermore, (\ref{eq.09-19.8}) ensures that $\beta$ is an $4\sqrt{\eta+5C_{|M|}\lambda}$ hyperlinear approximation to $G$. Therefore the decomposition $\beta \oplus \gamma$ as required to verify Lemma \ref{lem}. \end{proof}

\section{Proof of Theorem \ref{thm}}

Let $G$ be an amenable group, let $F$ be a finite subset of $G$ and let $\epsilon > 0$. Let $\kappa = \frac{\epsilon}{8|F|^2}$. We now perform a recursive construction. Let $L_1 = F$ and let $\lambda_1,\eta_1 > 0$ be small enough that \[ 2\eta_1+4\kappa+5C(|L_1|)\lambda_1 \leq \epsilon \]

\begin{description} \item[(Recursive hypothesis)] Let $n \in \mathbb{N}$. Suppose we have chosen numbers $\eta_1 > \cdots > \eta_n > 0$, finite subsets $L_1 \subseteq \cdots \subseteq L_n$ of $G$ and numbers $\lambda_1 > \cdots > \lambda_n > 0$ such that the following conditions are satisfied. \begin{description} \item[(Item A)] For all $j \in \{1,\ldots,n-1\}$ we have \[ \sqrt{ \eta_{j+1}+5 C_{|L_{j+1}|}\lambda_{j+1} } \leq \frac{\lambda_j}{8|L_j|^2} \] \item[(Item B)] For all $j \in \{1,\ldots,n\}$ we have \[ 2\eta_j+4\kappa+5C_{|L_j|}\lambda_j \leq \epsilon \]  \end{description} \item[(Extension procedure)] First choose $\eta_{n+1} > 0$ such that $\eta_{n+1} \leq \frac{\lambda_n^2}{2}$ and $2\eta_{n+1} \leq \frac{\epsilon}{4}$. Since $G$ is amenable, we can choose a finite subset $L_{n+1}$ of $G$ such that $|gL_{n+1} \cap L_{n+1}| \geq (1-\eta_{n+1})|L_{n+1}|$ for all $g \in L_n$. Then, choose $\lambda_{n+1} > 0$ such that   \[ 5C_{|L_{n+1}|}\lambda_{n+1} \leq \min\left(\frac{\lambda_n^2}{2} , \frac{\epsilon}{4}\right) \] \end{description}

We extend this procedure until $\left(1-\frac{\kappa}{2}\right)^n \leq \frac{\epsilon}{2}$. Choose $K = L_n$ and $\delta = \lambda_n$. Let $\alpha:G \to \mathrm{U}(X)$ be a $(K,\delta)$-hyperlinear approximation to $G$.\\
\\
We apply Lemma \ref{lem} iteratively. In the first step, we apply it to $\alpha$ with $M = L_n$, $L = L_{n-1}$, $\eta = \eta_n$ and $\lambda = \lambda_n$. In the $\ell^{\mathrm{th}}$ step we apply it to the space $Y_\ell$ and sofic approximation $\beta_\ell$ produced by the previous iteration. Item A ensures that each $\beta_\ell$ produced is an $(L_{n-\ell},\frac{\lambda_{n-\ell}}{8|L_{n-\ell}|^2})$-hyperlinear approximation to $G$ and therefore the next stage of the construction can proceed. Item B ensures that each $\gamma_\ell$ produced is a sofic-induced $(L_{n-\ell},\epsilon)$-hyperlinear approximation to $G$ and we have \[ \frac{1}{\Delta(Z_1)+\cdots+\Delta(Z_\ell)}\mathrm{tr}\Bigl(\bigl((\gamma_1\oplus \cdots \oplus \gamma_\ell)(g)-\alpha(g)\bigr)^\ast\bigl((\gamma_1\oplus \cdots \oplus \gamma_\ell)(g)-\alpha(g)\bigr) \Bigr) \leq \epsilon \] for all $g \in L_{n-\ell}$. Furthermore, the dimension of the underlying space of the remainder $Y_\ell$ is at most $\left(1-\frac{\kappa}{2}\right)\Delta(Y_{\ell-1})$ and so at the $n^{\mathrm{th}}$ stage we have $\Delta(Z_1\oplus \cdots \oplus Z_n) \geq (1-\epsilon)\Delta(X)$. Therefore we can choose $\omega$ to be $\gamma_1 \oplus \cdots \oplus \gamma_n$.

\bibliographystyle{plain}
\bibliography{stability.bib}

\end{document}